\documentclass[a4paper]{amsart}

\usepackage[T1]{fontenc}
\usepackage{tikz-cd}
\usepackage{orcidlink}
\usepackage[capitalise]{cleveref}
\usepackage{enumitem}
\usepackage{stmaryrd}
\usepackage{quiver}
\usepackage{mathtools}
\usepackage{libertine}
\usepackage{libertinust1math}
\usepackage{natbib}
\usepackage{amssymb}
\usepackage{url}

\newtheorem{theorem}{Theorem}[section]
\newtheorem*{theorem*}{Theorem}
\newtheorem{lemma}[theorem]{Lemma}
\newtheorem{corollary}[theorem]{Corollary}
\newtheorem{proposition}[theorem]{Proposition}

\theoremstyle{definition}
\newtheorem{example}[theorem]{Example}
\newtheorem{definition}[theorem]{Definition}
\newtheorem{remark}[theorem]{Remark}

\newcommand{\mc}[1]{\mathcal{#1}}
\newcommand{\mb}[1]{\mathbf{#1}}
\newcommand{\mbb}[1]{\mathbb{#1}}
\newcommand{\mr}[1]{\mathrm{#1}}

\newcommand{\Set}{\mb{Set}}

\newcommand{\Fin}{\mb{Fin}}

\newcommand{\sub}{\mr{Sub}}
\newcommand{\Cat}{\mb{Cat}}
\newcommand{\Mon}{\mb{Mon}}
\newcommand{\ProMon}{\mb{ProMon}}
\newcommand{\CTopoi}{\mb{CTopoi}}
\newcommand{\PTc}{\mb{PT}_c}
\newcommand{\PT}{\mb{PT}}
\newcommand{\Topoi}{\mb{Topoi}}

\newcommand{\psh}{\mb{Psh}}

\newcommand{\Mod}{\mb{Mod}}

\newcommand{\op}{^{\mathrm{op}}}
\newcommand{\inv}{^{\mathrm{-1}}}

\newcommand{\geo}[1]{\left|#1\right|}
\newcommand{\prth}[1]{\left(#1\right)}
\newcommand{\ov}[1]{\overline{#1}}

\newcommand{\set}[1]{\{\,#1\,\}}
\newcommand{\eff}{\Leftrightarrow}

\newcommand{\id}{\mathrm{id}}

\newcommand{\nt}{\Rightarrow}
\newcommand{\scomp}[2]{\{\,#1\,\mid\,#2\,\}}

\newcommand{\surj}{\twoheadrightarrow}
\newcommand{\inj}{\rightarrowtail}
\newcommand{\hook}{\hookrightarrow}

\newcommand{\N}{\mbb N}

\newcommand{\ex}[2]{\exists #1\!:\!#2.\,}
\newcommand{\fa}[2]{\forall #1\!:\!#2.\,}

\makeatletter
\newcommand{\ct@}[2]{%
  \vtop{\m@th\ialign{##\cr
    \hfil$#1\operator@font lim$\hfil\cr
    \noalign{\nointerlineskip\kern1.5\ex@}#2\cr
    \noalign{\nointerlineskip\kern-\ex@}\cr}}%
}
\newcommand{\ct}{%
  \mathop{\mathpalette\ct@{\rightarrowfill@\textstyle}}\nmlimits@
}
\makeatother
\makeatletter
\newcommand{\lt@}[2]{%
  \vtop{\m@th\ialign{##\cr
    \hfil$#1\operator@font lim$\hfil\cr
    \noalign{\nointerlineskip\kern1.5\ex@}#2\cr
    \noalign{\nointerlineskip\kern-\ex@}\cr}}%
}
\newcommand{\lt}{%
  \mathop{\mathpalette\lt@{\leftarrowfill@\textstyle}}\nmlimits@
}
\makeatother

\begin{document}

\title{Duality theory for categorical theories}

\author{Lingyuan Ye \orcidlink{0000-0002-8983-0099}}
\address{%
Lingyuan \textsc{Ye}\newline
Department of Computer Science and Technology\newline
University of Cambridge\newline
Cambridge, UK%
}

\begin{abstract}
We have generalised the notion of a categorical theory in model theory to the context of coherent theories. We prove a duality result between the full sub-2-category of pretopoi which are categorical, and the 2-category of profinite monoids. We also study the geometry of profinite monoids via the classifying topos construction, and show it identifies them as a full sub-2-category of the 2-category of topoi.\vspace{1ex}

\noindent
\textbf{Keywords.} duality, model theory, categoricity, coherent logic, profinite monoid.
\end{abstract}

\maketitle

\section{Introduction}

This paper characterises the space of models of a special class of first-order theories, which are called \emph{categorical} in model theory~\citep{hodges1997shorter}\footnote{We understand that the terminology is unfortunate for categorical logic. However, many important theorems in model theory, e.g. Morley's \emph{categoricity theorem}~\cite{morley1965categoricity}, use this terminology. Thus, in this paper we choose to keep in line with this tradition.}. For a classical first-order theory, it is categorical if all of its models are isomorphic to each other. Traditional model theory often ignores such theories because it almost exclusively studies theories of one sort. By the upward L\"owenheim–Skolem theorem (which is an easy consequence of the compactness theorem), if a theory has an infinite model, it has models of arbitrarily large cardinality. Thus, if a theory is categorical, then it can only have finite models. Though this claim still holds for multi-sorted theories (in an appropriate sense, cf.~\cref{lem:finitemodel}), as we will see, in this case the model theory of multi-sorted categorical theories is much more interesting, since they possess highly non-trivial automorphism groups.

In fact, in this paper we will not follow a syntactic approach towards first-order theories as model theory usually does. We will adopt the \emph{categorical logic} approach, which identifies first-order theories as \emph{pretopoi}, cf.~\citet[D1]{johnstone2002sketches}. Thus, more precisely speaking, for us a first-order theory will be a coherent theory. Classical first-order theories then correspond to \emph{Boolean} pretopoi, where the subobject lattice of any object has complements, viz.\ forms a Boolean algebra. In this case, the notion of a categorical theory needs to be slightly modified to account for the lack of negation in the syntax. We will call a coherent theory \emph{categorical} if all of its models are a \emph{retract} of a particular model (cf.~\cref{def:categorical}).

The main result of this paper can be summarised as follows. We write $\PTc$ for the full sub-2-category of pretopoi that are categorical. For any $\mc C\in\PTc$, its category of models $\Mod(\mc C)$ will be equivalent to the Cauchy completion of a unique monoid, viz.\ a category with one object. In this case, the monoid has a canonical \emph{profinite topology}, making it into a profinite monoid. This induces an equivalence of 2-categories:

\begin{theorem*}[\ref{thm:main}]
    The 2-category of categorical pretopoi is dually equivalent to the 2-category of profinite monoids, and this equivalence restricts to one between Boolean categorical pretopoi and profinite groups,
    \[
    \begin{tikzcd}
    \mb{BPT}_c\op \ar[r, "\simeq"] \ar[d, hook] & \mb{ProGrp} \ar[d, hook] \\
    \PTc\op \ar[r, "\simeq"] & \ProMon.
    \end{tikzcd}
    \]
\end{theorem*}

We emphasise here that the above duality is 2-categorical in nature. The 2-categorical structure on $\PTc$ is induced by that on the 2-category $\PT$ of all pretopoi, i.e.\ the 1-cells are coherent functors and 2-cells are natural transformations between them. The 1-cells in $\ProMon$ are \emph{continuous semigroup homomorphisms}, and 2-cells are \emph{conjugates} between them. In the remaining part of the introduction, we will briefly discuss the motivation for such a duality result from both a logical and geometric perspective, and also summarise the main contents of the remaining sections.

\subsection{A Logical Motive: Spectrum of First-Order Theories}

From a logical perspective, the main motivation for obtaining such a result takes root in the general study of the duality theory for first-order logic. G\"odel's completeness theorem states that the \emph{category} of models of a first-order theory can \emph{separate} the syntax of a theory, viz.\ if two formulas are interpreted identically over all models, then they are provably equivalent over this theory. A duality result, furthermore, asks for \emph{reconstruction} of the original theory from its models. The central question is what additional information on the category of models allows us to recover the original first-order theory up to equivalence.

In the propositional case, we have the Stone and Priestley dualities that answer this problem. More precisely, for a distributive lattice, the additional information on its poset of models that allows us to reconstruct it consists of a certain \emph{topology}, which is called the Priestley topology. Any distributive lattice determines a Priestley topology on its poset of models, and the original distributive lattice is recoverable from this topological poset up to isomorphism.

For first-order theories, \citet{makkai1982stone} proved his celebrated \emph{conceptual completeness} result, which shows that a pretopos is recoverable by the \emph{ultracategorical} structure on its category of models. Roughly speaking, the ultracategory structure records the operation of ultraproducts of models, which in some sense is a categorification of convergence of ultranets in a topological space. \citet{lurie2018ultracategories} has then reformulated the theory of ultracategories \`a la Makkai and provided an alternative proof for conceptual completeness. Following his efforts, more recently there has been a proliferation of works on ultracategories~\citep{VANGOOL2026108269,saadia2025extending,hamad2025generalised} that are much more categorical in nature. From a quite different perspective, \citet{di2022geometry} describes the ultrastructures of models of pretopoi as a \emph{Kan-injectivity} condition in the 2-category of topoi. Extending this idea,~\citet{di2025logic} have provided similar characterisations for other subfragments of geometric logic, and~\cite{DiLibertiIvan2026Ccfs} have proven conceptual completeness of coherent logic, among other fragments of geometric logic, in this framework.

However, unlike Stone or Priestley duality for propositional theories where the additional structures on posets of models are described by topological data, the notion of an ultracategory is much harder to work with. In most of the previously mentioned approaches towards ultracategories, an ultracategory in general involves a \emph{large} amount of data, i.e.\ the ultrastructure specifies for \emph{each} ultrafilter $\mu$ on \emph{each} set $X$ how to take the ultraproduct with respect to an $X$-indexed family of models. However, the ultracategories we care about in practice intuitively involve only a \emph{small} amount of information, since the ultrastructure should be completely determined by the syntax, which is encoded by a small pretopos. For this we would also like to mention the approach of~\citet{breiner2013duality}, which constructs the spectrum of a first-order theory as a stack of \emph{small} pretopoi on a certain topological space, akin to the spectrum construction in algebraic geometry.

The original goal of this paper is to find a special class of pretopoi, where the additional ultrastructure on their categories of models can in fact be described again by a suitable \emph{topology}. Of course, any distributive lattice can be viewed as a first-order theory, thus we are not interested in propositional examples. A non-trivial example has already been shown in~\citet{ye2024stack}, which exhibits a duality result for \emph{essentially finite} theories. A theory is essentially finite if its category of models is equivalent to a \emph{finite} category. There is a dual equivalence between Cauchy complete finite categories and essentially finite first-order theories. However, in this case the ultrastructure is trivial, since any ultraproduct on a finite family of points simply selects some model in that family. In other words, the ultrastructure on a finite Cauchy complete category is given by the \emph{discrete topology}.

From this perspective, here we have established the first non-trivial example of a class of first-order theories, such that there exist genuine topological structures on their categories of models that allow for a dual reconstruction theorem. This class is the collection of categorical coherent theories, and the topologies are certain profinite topologies. We will introduce the correct notion of categorical pretopos in~\cref{sec:categoricaltheory}, and study some of its consequences. \cref{sec:profmonoidofcattheory} will then explain the associated profinite monoid for a categorical pretopos. Our proof that this construction gives an equivalence of 2-categories is based on studying the \emph{classifying topoi} of profinite monoids, which connects to another general motivation of this paper.

\subsection{A Geometric Motive: Classifying Topoi of Profinite Monoids and Semi-Galois Theory}

The results in this paper might be independently interesting to the readers without a logical motivation. In particular, for a profinite monoid $\mc M$, its \emph{classifying topos} $\Set[\mc M]$ can be viewed as the category of continuous left $\mc M$-sets. We will recall in~\cref{sec:toposofprofmonoid} that the classifying topos of a profinite monoid $\mc M$ is a coherent topos. We will show that this construction is 2-functorial, i.e.\ it forms a 2-functor $\ProMon \hook \CTopoi$. Here $\CTopoi$ is the 2-category of coherent topoi, with 1-cells being coherent geometric morphisms, i.e.\ those whose inverse images preserve coherent objects. 

In~\cref{sec:duality}, we will show that this 2-functor is \emph{2-fully faithful}. In fact, we will show this 2-functor is also 2-fully faithful when postcomposed with the (non-full) inclusion $\CTopoi \inj \Topoi$, i.e.\ any geometric morphism between topoi of the form $\Set[\mc M]$ for profinite monoids $\mc M$ is automatically coherent (cf.~\cref{thm:fullyfaithfulpromonoidtocoh}). This shows we can identify profinite monoids as a full sub-2-category of the 2-category of topoi,
\[ \ProMon \hook \Topoi \]
and thus implies that $\Topoi$ is a good place to study the geometry of profinite monoids, and also the associated logical structures of categorical theories. This result alone suffices to characterise the category of points of the classifying topos $\Set[\mc M]$ of profinite monoids (cf.~\cref{cor:promoncategorical}). We will use this characterisation to finally prove the main duality theorem in~\cref{sec:duality_for_categorical_theories}. 

Although semi-Galois theory in the sense of~\citet{URAMOTO2025107863} is not the main focus of this paper, these results have a deep connection to it. In fact, \emph{loc.\ cit.}\ has introduced the notion of semi-Galois categories, and established a similar duality result between them and profinite monoids. The results in this paper then provide a new characterisation of semi-Galois categories, i.e.\ they are exactly categorical pretopoi (equipped with the canonical model). Dually, this also provides an alternative characterisation of classifying topoi of profinite monoids: they are exactly the categorical coherent topoi, viz.\ those coherent topoi having a point of which every point is a retract.

We mention that the duality result established in this paper is a two-fold generalisation of the one obtained in \emph{loc.\ cit.} Firstly, we consider a wider class of maps between profinite monoids, viz.\ continuous \emph{semigroup} morphisms, which do not necessarily preserve the identity. From the perspective of the geometry of Galois theory, this means our result provides a duality theory accounting for the maps that do not necessarily preserve the base points of the underlying spaces. From a logical perspective, this means we consider all coherent functors between the categorical pretopoi, instead of those that preserve the canonical model. Secondly, our duality is \emph{2-categorical} in nature, where the 2-categorical structure is important from the logical perspective, since it accounts for natural transformations between coherent functors, or equivalently morphisms between models. We believe that this perspective will provide a bridge between model theory and semi-Galois theory. This in turn has a deep connection with the theory of regular languages and finite automata, which is explained in~\citet{URAMOTO2025107863}.

\subsection*{Conventions}

In this paper we assume familiarity with the basic language of categorical logic and topos theory. We will freely identify a pretopos $\mc C$ as a coherent theory, whose objects are thought of as formulas and whose models are coherent functors from $\mc C$ to $\Set$. Moreover, if $\varphi,\psi$ are subobjects of $\chi$ in $\mc C$, then for any model $M$ of $\mc C$ we will write $M \models \varphi \vdash \psi$, if $M\varphi \subseteq M\psi$ as subsets of $M\chi$. For instance, given a mono $\varphi \inj \psi$, by viewing it as a subobject of $\psi$, $M \models \psi \vdash \varphi$ iff $M$ maps it to an isomorphism. More generally, we will also write $\mc C \models \psi \vdash \varphi$, if the mono $\varphi \inj \psi$ is an isomorphism in $\mc C$. \cite{johnstone2002sketches},~\cite{maclane2012sheaves}, and~\cite{caramello2018theories} are good references for categorical logic and topos theory. We will also rely on some basic results in model theory, of which \citet{hodges1997shorter} gives a textbook account.

For a monoid $\mc M$ and $a\in\mc M$, we write $a\mc M \coloneq \scomp{ab}{b\in \mc M} \subseteq \mc M$ as the right ideal generated by $a$; similarly for the left ideal $\mc Ma$. If $a$ is \emph{idempotent} in $\mc M$, viz.\ $aa = a$, then $a\mc Ma$ is a bi-ideal in $\mc M$. For some basic theory of monoids, see~\citet{Steinberg2016}.

\section{The Structure of Categorical Theories}\label{sec:categoricaltheory}

In traditional model theory, a theory is \emph{categorical} if it only has one model up to isomorphism. This is suitable for classical first-order theories because when the language contains negation, any morphism between models, viz.\ any elementary embedding, will automatically be injective, since it needs to preserve both the interpretation of equality and inequality. However, in the context of coherent logic, this requirement turns out to be too strong. As a first observation, a retract of a model of a pretopos is also a model:

\begin{lemma}\label{lem:retractpreservestruth}
  Let $M,N : \mc C \to \Set$ be functors such that $N$ is a retract of $M$ in $[\mc C,\Set]$. If $M$ preserves a (co)limit in $\mc C$, then so does $N$. In particular, if $\mc C$ is a pretopos and $M$ is coherent, then so is $N$. In this case, for any mono $\psi \inj \varphi$ in $\mc C$,
  \[ M \models \varphi \vdash \psi \nt N \models \varphi \vdash \psi. \]
\end{lemma}
\begin{proof}
  Let $r : M \leftrightarrows N : s$ be the retract pair between $M$ and $N$. For any diagram $\varphi : I \to \mc C$, suppose $\lt_{i\in I}\varphi_i$ exists in $\mc C$. By naturality the following diagram commutes,
    \[\begin{tikzcd}[row sep = huge]
    	{\lt_{i\in I}M\varphi_i} & {M\lt_{i\in I}\varphi_i} \\
    	{\lt_{i\in I}N\varphi_i} & {N\lt_{i\in I}\varphi_i}
    	\arrow["{\lt_{i\in I}r_i}"', curve={height=12pt}, two heads, from=1-1, to=2-1]
    	\arrow["{\alpha_M}"', from=1-2, to=1-1]
    	\arrow["r"', curve={height=12pt}, two heads, from=1-2, to=2-2]
    	\arrow["{\lt_{i\in I}s_i}"', curve={height=12pt}, tail, from=2-1, to=1-1]
    	\arrow["s"', curve={height=12pt}, tail, from=2-2, to=1-2]
    	\arrow["{\alpha_N}", from=2-2, to=2-1]
    \end{tikzcd}\]
    Here $\alpha_M,\alpha_N$ are the canonical maps induced by the limit. If $M$ preserves this limit, $\alpha_M$ has an inverse $\alpha_M\inv$. We may construct the inverse of $\alpha_N$ as follows,
    \[ \alpha_N\inv \coloneq r \circ \alpha_M\inv \circ \lt_{i\in I}s_i. \]
    We observe this is indeed an inverse,
    \begin{align*}
        \alpha_N\inv \circ \alpha_N 
        &= r \circ \alpha_M\inv \circ \lt_{i\in I}s_i \circ \alpha_N \\
        &= r \circ \alpha_M\inv \circ \alpha_M \circ s \\ 
        &= r \circ s \\
        &= \id
    \end{align*}
    and that 
    \begin{align*}
        \alpha_N \circ \alpha_N\inv
        &= \alpha_N \circ r \circ \alpha_M\inv \circ \lt_{i\in I}s_i \\
        &= \lt_{i\in I}r_i \circ \alpha_M \circ \alpha_M\inv \circ \lt_{i\in I}s_i \\ 
        &= \lt_{i\in I}r_i \circ \lt_{i\in I}s_i \\
        &= \id
    \end{align*}
    The case for colimits is completely similar. If $\mc C$ is a pretopos and $M$ is a coherent functor, recall that $M \models \varphi \vdash \psi$ iff $M\psi \inj M\varphi$ is an isomorphism. Since isomorphisms are closed under retracts, it follows that $N \models \varphi \vdash \psi$ as well.
\end{proof}

For a non-Boolean pretopos, in general there will be non-trivial retracts of models. Hence, the more appropriate notion of categoricity would be as follows:

\begin{definition}[Categorical coherent theory]\label{def:categorical}
  A pretopos $\mc C$ is \emph{categorical}, if $\mc C$ has one model up to retract, i.e.\ there exists a model $\star$ where every model is a retract of $\star$.
\end{definition}

For a categorical theory $\mc C$, $\star$ will be called the \emph{canonical model} of $\mc C$. However, this model will \emph{not} be canonical in the sense of having a universal property. We will see in~\cref{prop:canmodelunique} that indeed two canonical models of a categorical theory will be \emph{isomorphic} to each other, but such an isomorphism is by no means unique.

The aim of the remainder of this section is to establish various properties of categorical pretopoi, using elementary tools from model theory. Our first observation is that the canonical model $\star$ of $\mc C$ contains all the logical information of $\mc C$, in the sense that it is a \emph{conservative} model:

\begin{lemma}
  For a categorical pretopos $\mc C$, the canonical model $\star : \mc C \inj \Set$ is a conservative functor.
\end{lemma}
\begin{proof}
  Coherent logic is complete, in the sense that the family of $\Set$-models is jointly conservative for any small pretopos. This is a consequence of Deligne's theorem that all coherent topoi have enough points~\citep{SGA4}; also cf.~\cite[D1.5.10]{johnstone2002sketches}. Hence, a mono $\psi \inj \varphi$ in $\mc C$ is an isomorphism iff for all models $M$ of $\mc C$, $M \models \varphi \vdash \psi$. By~\cref{lem:retractpreservestruth}, this happens iff $\star \models \varphi \vdash \psi$, i.e.\ $\star\psi \inj \star\varphi$ is an isomorphism. Thus, $\star$ is conservative.
\end{proof}

Another important observation is that the canonical model will be \emph{finite} in an appropriate sense. Since in general we consider \emph{multi-sorted} theories, we introduce the relevant finiteness condition as follows:

\begin{definition}[Sortwise finite models]
    A model $M$ of a small pretopos $\mc C$ is \emph{sortwise finite}, if as a functor $M : \mc C \to \Set$ it factors through the inclusion $\Fin \hook \Set$ of finite sets. In other words, $M\varphi$ is finite for all $\varphi\in\mc C$.
\end{definition}

\begin{lemma}\label{lem:finitemodel}
  The canonical model $\star$ of a categorical pretopos $\mc C$ is sortwise finite.
\end{lemma}
\begin{proof}
  By the \emph{upward L\"owenheim–Skolem theorem} (which is an easy consequence of the \emph{compactness theorem}), if $\star\varphi$ for some $\varphi\in\mc C$ is infinite, then there exist models $M$ of $\mc C$ such that $M\varphi$ has arbitrarily large cardinality. In particular, not all such models can be retracts of $\star$ since the cardinality of $\star\varphi$ is bounded. Hence, $\star\varphi$ must be a finite set.
\end{proof}

\begin{corollary}\label{categoricaltheoryhasfinitemodel}
    For a categorical pretopos, all of its models are sortwise finite.
\end{corollary}
\begin{proof}
    Retracts of sortwise finite models are sortwise finite.
\end{proof}

\begin{remark}[Theories with only sortwise finite models]
    It is not true that if a pretopos $\mc C$ only has sortwise finite models, then $\mc C$ will be categorical. For instance, if $\mc C$ is induced by a propositional theory, i.e.\ $\mc C$ is the pretopos completion of a distributive lattice, then models of $\mc C$ will also be sortwise finite. 
\end{remark}

As an immediate consequence, we show that the canonical model of a categorical theory $\mc C$ is unique up to isomorphism (though not up to unique isomorphism):

\begin{proposition}\label{prop:canmodelunique}
  Let $\mc C$ be a categorical pretopos. If there are two models $M,N$ of $\mc C$ such that all models are both retracts of $M$ and of $N$, then $M,N$ are isomorphic.
\end{proposition}
\begin{proof}
  By assumption, $M$ is a retract of $N$, and $N$ is a retract of $M$. Composing the two sections we get an endomorphism $M \inj M$, which is pointwise injective. For any $\varphi\in\mc C$, this induces an injection $M\varphi \inj M\varphi$. But since $M\varphi$ is finite by~\cref{lem:finitemodel}, and any injective endomorphism on a finite set must be bijective, the map $M \inj M$ must be an isomorphism. Similarly for $N$. Thus, $M,N$ are isomorphic.
\end{proof}

Using these results, we show that categorical pretopoi have many special properties:

\begin{proposition}\label{prop:strucofcat}
  Let $\mc C$ be a categorical pretopos. Then we have:
  \begin{enumerate}
    \item It is two-valued, i.e.\ $\sub_{\mc C}(1) \cong 2$;
    \item It is of finite type, i.e.\ $\sub_{\mc C}(\varphi)$ is finite for any $\varphi\in\mc C$;
    \item It is locally finite, i.e.\ $\mc C(\varphi,\psi)$ is a finite set for any $\varphi,\psi\in\mc C$;
    \item It admits finite colimits, and $\star$ preserves them.
  \end{enumerate}
\end{proposition}
\begin{proof}
  (1): Since $\star$ is conservative, $\sub(1)$ must be a subposet of $\mc P(1) \cong 2$. Since $\mc C$ has a model, it cannot be degenerate, hence $\sub_{\mc C}(1) \cong 2$. 

  (2): Similarly, $\sub_{\mc C}(\varphi)$ must be a subposet of $\mc P(\star\varphi)$, which is finite.

  (3): A conservative coherent functor must also be \emph{faithful}, thus $\mc C(\varphi,\psi)$ is a subset of $\Fin(\star\varphi,\star\psi)$, which is finite.

  (4): It suffices to show $\mc C$ has pushouts and they are preserved by $\star$. Since any pretopos is \emph{adhesive} (cf.~\cite{garner2012remarks}), pushouts of monos always exist and are preserved by any coherent functor. Hence, it suffices to show pushouts of epis exist. For an epi $q : \varphi \surj \psi$ and a map $f : \varphi \to \chi$, let $R \rightrightarrows \varphi$ be the kernel pair of $q$. Since $\sub(\chi \times \chi)$ is finite, there is a smallest equivalence relation $Q$ on $\chi$ that the composite $R \inj \varphi \times \varphi \to \chi \times \chi$ factors through $Q$. It is then routine to show that the effective coequaliser of $Q \rightrightarrows \chi$ is the pushout of $q$ along $f$. Furthermore, from this description it is evident that such pushouts are preserved by coherent functors, since they preserve finite limits, finite unions of subobjects, and effective quotients.
\end{proof}

\begin{remark}[Categorical pretopoi as semi-Galois categories]
  \citet{URAMOTO2025107863} has introduced a notion of \emph{semi-Galois category}, which is a pair $(\mc C, F)$ where $\mc C$ is a category and $F : \mc C \to \Fin$ is a functor, such that: 
  \begin{itemize}
      \item $\mc C$ has finite limits, finite colimits, and image factorisation;
      \item $F$ preserves finite limits and finite colimits, and is conservative;
  \end{itemize}
  By~\cref{prop:strucofcat}, for any categorical pretopos $\mc C$, $(\mc C,\star)$ forms a semi-Galois category in this sense.
\end{remark}

\begin{remark}[Categoricity in model theory]
  In model theory, a better studied notion of categoricity is that of \emph{$\kappa$-categoricity} for some infinite cardinal $\kappa$. A theory is $\kappa$-categorical if up to isomorphism it only has one model of cardinality $\kappa$. By the Ryll–Nardzewski theorem, a complete $\omega$-categorical theory in a countable language has finite types as well~\citep{hodges1997shorter}. In the future we will study these more general theories from a duality perspective. 
\end{remark}

We can indeed characterise the entire category of models of a categorical theory. Recall that any small category $\mc C$ has a \emph{Cauchy completion} $\ov{\mc C}$, which is the universal category where every idempotent in $\mc C$ splits. Concretely, $\ov{\mc C}$ can be identified as the full subcategory of $\psh(\mc C)$ whose objects are retracts of representables. This way, we will show that a pretopos is categorical iff its category of models is the Cauchy completion of the classifying category of a monoid.

\begin{definition}[Classifying category of a monoid]
    For any monoid $\mc M$, we use $B\mc M$ to denote its \emph{classifying category}, which is a one-object category with endomorphism monoid $\mc M$, where the composition is given by the product in $\mc M$. The unique object in $B\mc M$ will be denoted as $\star$. 
\end{definition}

\begin{proposition}\label{prop:monoiduptocauchy}
  A pretopos $\mc C$ is categorical iff there is a monoid $\mc M$ that $\Mod(\mc C) \simeq \ov{B\mc M}$.
\end{proposition}
\begin{proof}
  By construction, every object in $\ov{B\mc M}$ is a retract of $\star$. Thus if $\Mod(\mc C) \simeq \ov{B\mc M}$, then $\mc C$ will be categorical by definition. On the other hand, if $\mc C$ is categorical, let $\mc M$ be the endomorphism monoid of its canonical point $\star$, 
  \[ \mc M \coloneq \Mod(\mc C)(\star,\star). \]
  $\Mod(\mc C)$ is Cauchy complete by~\cref{lem:retractpreservestruth}, thus the embedding $B\mc M \hook \Mod(\mc C)$ induces a fully faithful inclusion $\ov{B\mc M} \hook \Mod(\mc C)$. Since $\mc C$ is categorical, this functor is also essentially surjective, thus we have an equivalence $\ov{B\mc M} \simeq \Mod(\mc C)$.
\end{proof}

\begin{remark}[A concrete description of the Cauchy completion of a monoid]\label{rem:Cauchyofmonoid}
    For any monoid $\mc M$, the Cauchy completion $\ov{B\mc M}$ of its classifying category can be concretely described as follows:
    \begin{itemize}
        \item Objects are of the form $\star_e$ where $e\in\mc M$ is idempotent, i.e.\ $ee = e$;
        \item A morphism $f : \star_e \to \star_d$ is an element $f\in\mc M$ such that $fe = f = df$. 
    \end{itemize}
    The composition of morphisms in $\ov{B\mc M}$ is inherited from the product of $\mc M$, and $B\mc M$ is included in $\ov{B\mc M}$ by sending $\star$ to $\star_1$. For any idempotent $e$, $\ov{B\mc M}(\star_e,\star_e) \cong e\mc Me$.
\end{remark}

As a special case, the category of models of a \emph{Boolean} categorical pretopos will be equivalent to the classifying category of a group:

\begin{corollary}\label{cor:booleancategorical}
  For a Boolean categorical pretopos $\mc B$, $\Mod(\mc B) \simeq B\mc G$ for a group $\mc G$.
\end{corollary}
\begin{proof}
  As mentioned at the start of~\cref{sec:categoricaltheory}, if $\mc B$ is Boolean then all morphisms between models of $\mc B$ must be \emph{injective}. Since the canonical model $\star : \mc B \inj \Fin$ is sortwise finite, any injective endomorphism between finite sets is automatically a bijection, which implies the endomorphism monoid of $\star$ in fact forms a group $\mc G$. Furthermore, $B\mc G$ is already Cauchy complete, since an idempotent element in $\mc G$ is necessarily the identity. This way,~\cref{prop:monoiduptocauchy} implies we have an equivalence $\Mod(\mc B) \simeq \ov{B\mc G} \simeq B\mc G$.
\end{proof}

\begin{remark}[Theories with a groupoid of models]
    After establishing the duality result in~\cref{sec:duality_for_categorical_theories}, we will see that~\cref{cor:booleancategorical} in fact describes an if and only if condition, i.e.\ a categorical pretopos is Boolean iff its category of models is equivalent to the classifying category of a group. This raises an interesting question of when the category of models of a coherent theory is a \emph{group(oid)}. Here we only observe that if $\Mod(\mc C)$ is a groupoid for a pretopos $\mc C$, then again by the upward L\"owenheim–Skolem theorem all models of $\mc C$ must be sortwise finite. If a model $M$ of $\mc C$ interprets $\varphi\in\mc C$ as an infinite set $M\varphi$, then there exists an embedding $M \inj M'$ such that $M'\varphi$ has strictly larger cardinality. In the context of topos theory, this corresponds to \emph{grouplike} topoi~\cite[C3.6]{johnstone2002sketches}; they have also been briefly studied in~\citet{di2025logic}.
\end{remark}

We end this section by exhibiting some non-trivial examples of categorical theories. For single-sorted theories, the fact that we are considering more generally coherent theories rather than just \emph{classical} first-order theories already provides us with more examples:

\begin{example}\label{exm:cohsimple}
    Consider a single-sorted language with no additional non-logical symbols. Let $T_2$ denote the theory with the following coherent axioms,
    \[ T_2 \coloneq \set{\exists x.\,\top} \cup \set{\forall xyz.\, x = y \vee x = z \vee y = z}. \]
    A model of $T_2$ is a set with either a single element or two distinct elements. Hence, every model of $T_2$ is a retract of $2 = \set{0,1}$, so it is categorical in the sense of~\cref{def:categorical}. The endomorphism monoid of the canonical model of $T_2$ can thus be identified with the endomorphism monoid of the set $2$.
\end{example}

On the other hand, the \emph{classical} first-order theory $T^c_2$ with the same axioms as $T_2$ will \emph{not} be categorical, since with negation, finite cardinality is definable in the classical language. In particular, for a single-sorted categorical classical theory, all models will have some fixed finite cardinality $n$. The following example then shows how more interesting structures can arise when one considers theories with infinitely many sorts:

\begin{example}\label{exm:simple}
  Consider a language with countably many sorts $\set{A_n}_{n\in\N}$ but with no additional non-logical symbols. Let $T^c_N$ be a classical first-order theory in this language, where for each sort $A_n$ it has an axiom stating that the sort has exactly $n$ elements:
  \[ \ex{x_1,\cdots,x_n}{A_n} \prth{\bigwedge_{i\neq j}x_i \neq x_j} \wedge \fa x{A_n} \bigvee_{1\le i\le n}x = x_i. \]
  The theory $T^c_N$ has exactly one model $\star$ up to isomorphism, where $\star$ interprets the sort $A_n$ as an $n$-element set. However, the automorphism group of this model is highly non-trivial. Since any permutation of an $n$-element set induces an isomorphism on the interpretation at $A_n$, the automorphism group of $\star$ will be the following product,
  \[ \mr{Aut}(\star) \cong \prod_{n\in\N} S_n. \]
  Since each $S_n$ is a finite group, the product $\mr{Aut}(\star)$ has a natural \emph{profinite topology}, which is the product topology, making it into a profinite group.
\end{example}

We have already shown that a categorical pretopos $\mc C$ uniquely determines a monoid $\mc M$, such that $\Mod(\mc C) \simeq \ov{B\mc M}$. However, just like the case for propositional logic, the categorical structure of models does not provide a full reconstruction for the logical theory. In the next section we will see that the monoid $\mc M$ is naturally equipped with a \emph{profinite topology}. The duality result then shows $\mc C$ is determined by $\mc M$ as a profinite monoid.

\section{The Profinite Monoid for a Categorical Theory}\label{sec:profmonoidofcattheory}

For this section, let us fix a categorical pretopos $\mc C$ with the canonical model $\star$, whose endomorphism monoid is $\mc M$. The goal of this section is to observe that $\mc M$ is equipped with a natural profinite topology, which can be deduced from the sortwise finiteness of the canonical model $\star$. Let us first consider the following product space
\[ \mc S \coloneq \prod_{\varphi\in\mc C}\Fin(\star\varphi,\star\varphi). \]
An element $f$ in $\mc S$ is a family of functions $(f_\varphi)_{\varphi\in\mc C}$ indexed by objects in $\mc C$. $\mc S$ has a natural product topology which is profinite, viz.\ compact Hausdorff and totally disconnected, since it is a product of finite discrete spaces. We will denote the projection maps as 
\[ \pi_\varphi : \mc S \surj \Fin(\star\varphi,\star\varphi) \] 
for any object $\varphi\in\mc C$. For any subset $U \subseteq \prod_{i\in n}\Fin(\star\varphi_i,\star\varphi_i)$, its inverse image under the projection maps by construction is clopen,
\[ \mc S_{U} \coloneq (\prod_{i\in n}\pi_{\varphi_i})\inv(U) = \scomp{f\in\mc S}{(f_{\varphi_i})_{i\in n}\in U}. \]
In particular, for any function $g : \star\varphi \to \star\varphi$ for some $\varphi\in\mc C$, $\mc S_g \coloneq \mc S_{\set g}$ is clopen in $\mc S$, and these subspaces form a subbasis of the profinite topology on $\mc S$.

In this case, there is a natural inclusion $\mc M \subseteq \mc S$ of the endomorphism monoid $\mc M$ of $\star$ into the product space $\mc S$, where $\mc M$ carves out those families of functions satisfying the naturality condition for $\mc C$. This makes $\mc M$ into a profinite monoid:

\begin{lemma}\label{prop:profinitemonoid}
  The inclusion $\mc M \subseteq \mc S$ identifies $\mc M$ as a closed subspace of $\mc S$, thus $\mc M$ also has a profinite topology. In this topology, $\mc M$ is a profinite monoid.
\end{lemma}
\begin{proof}
    An element $f$ in $\mc M$ by definition is an element in $\mc S$ satisfying the naturality condition for all morphisms in $\mc C$. The naturality condition on a morphism $\alpha : \varphi \to \psi$ in $\mc C$ is indeed \emph{local} in the sense that it only involves $\varphi$ and $\psi$. Thus, we may define the clopen subspace $\mc S_\alpha$ as follows,
    \[ \mc S_\alpha \coloneq \scomp{f\in\mc S}{\star\alpha \circ f_\varphi = f_\psi \circ \star\alpha}, \]
    and $\mc M = \bigcap_{\alpha\in\mc C_1}\mc S_\alpha$, where $\mc C_1$ denotes the set of morphisms in $\mc C$, will thus be a \emph{closed} subspace of $\mc S$, hence has a profinite topology.

    To show $\mc M$ is a profinite monoid, we need to show the multiplication on $\mc M$, which is given by composition $\circ$, is continuous. Since subspaces of the form $\mc S_g$ for some $g : \star\varphi \to \star\varphi$ form a subbasis of $\mc S$, subspaces of the form $\mc M_g \coloneq \mc M \cap \mc S_g$ also form a subbasis of $\mc M$. Hence, it suffices to consider their preimage under $\circ$. Given such a $g$, by definition
    \[ \circ\inv(\mc M_g) = \scomp{(f,h) \in \mc M \times \mc M}{f_\varphi \circ h_\varphi = g} = \bigcup_{u \circ v = g} \mc M_u \times \mc M_v. \]
    Since there are only finitely many possible decompositions $u \circ v = g$ as $\Fin(\star\varphi,\star\varphi)$ is finite, the above union is finite, thus $\circ\inv(\mc M_g)$ is again clopen in $\mc M \times \mc M$, which implies $\circ$ is continuous for the profinite topology.
\end{proof}

The construction $\mc C \mapsto \mc M \coloneq \Mod(\mc C)(\star,\star)$ thus associates a categorical pretopos with the profinite endomorphism monoid of its canonical model. The remaining goal of this section is to discuss functoriality of this construction. Suppose we have two categorical theories $\mc C,\mc D$, with associated profinite monoids $\mc M,\mc N$, respectively. Then any coherent functor $\theta : \mc D \to \mc C$ will induce a functor by precomposing with $\theta$,
\[ -\circ\theta : \Mod(\mc C) \to \Mod(\mc D). \]
Since we have equivalences $\Mod(\mc C) \simeq \ov{B\mc M}$ and $\Mod(\mc D) \simeq \ov{B\mc N}$ by~\cref{prop:monoiduptocauchy}, this equivalently induces a functor
\[ -\circ\theta : \ov{B\mc M} \to \ov{B\mc N}. \]
We thus first provide a concrete description of the category of functors between Cauchy completions of the classifying categories of two monoids.

Given monoids $\mc M,\mc N$, recall that a \emph{semigroup morphism} $f : \mc M \to \mc N$ is a map between them that preserves multiplication, but not necessarily the unit. A \emph{conjugate} $a : f \nt g$ between semigroup morphisms $f,g : \mc M \to \mc N$ is an element $a\in \mc N$, such that $af(1) = a = g(1)a$, and $a f(m) = g(m) a$ for any $m\in\mc M$. These data organise into a 2-category, which we denote as $\Mon$.

\begin{lemma}\label{lem:cauchyfunc}
  For monoids $\mc M,\mc N$, we have an equivalence of categories,
  \[ \Cat(\ov{B\mc M},\ov{B\mc N}) \simeq \Mon(\mc M,\mc N). \]
\end{lemma}
\begin{proof}[Proof Sketch]
  Here we provide the data of such an equivalence; for a complete proof showing it is an equivalence, see~\citet[Prop. 1.4.3]{rogers2021toposes}. By the universal property of Cauchy completion, we have
  \[ \Cat(\ov{B\mc M},\ov{B\mc N}) \simeq \Cat(B\mc M,\ov{B\mc N}). \]
  By the description of the Cauchy completion $\ov{B\mc N}$ in~\cref{rem:Cauchyofmonoid}, a functor $B\mc M \to \ov{B\mc N}$ can then be equivalently described as a pair $(e,f)$, where $e\in \mc N$ is an idempotent element and $f : \mc M \to e\mc Ne$ is a monoid morphism. It is well-known that this is exactly the data of a semigroup map $f : \mc M \to \mc N$, with $e = f(1)$; see e.g.~\citet{Steinberg2016}. By the description of morphisms in $\ov{B\mc N}$, it is also easy to see that natural transformations correspond to conjugates.
\end{proof}

Let us then consider the 2-category $\ProMon$ of profinite monoids, where now morphisms are given by \emph{continuous} semigroup morphisms and 2-cells are again conjugates between them. Also let $\PTc$ be the 2-category of categorical pretopoi. The association of a profinite monoid to a categorical pretopos can be made functorial as follows. We assume that for any $\mc C\in\PTc$ we have chosen a canonical model $\star_{\mc C}$. Furthermore, we also assume the choice of an explicit inverse functor
\[ I_{\mc C} : \Mod(\mc C) \simeq \ov{B\mc M}, \]
where $\mc M \coloneq \Mod(\mc C)(\star_{\mc C},\star_{\mc C})$ is the associated profinite monoid of $\mc C$, for any $\mc C\in\PTc$. Explicitly, $I_{\mc C}$ specifies for any model $M$ of $\mc C$ a retract pair $r_M : \star_{\mc C} \leftrightarrows M : s_M$. Without loss of generality, we may assume the retract pair is a pair of identities when $M = \star_{\mc C}$.

\begin{theorem}
  The above data induce a pseudo-functor
  \[ P : \PTc\op \to \ProMon. \] 
\end{theorem}
\begin{proof}
  On objects, $P$ sends a categorical pretopos to its associated profinite monoid. Given two categorical theories $\mc C,\mc D$, whose associated profinite monoids are $\mc M,\mc N$ respectively, the choice of $I_{\mc D}$ induces a functor
  \[ P_- : \PTc(\mc D,\mc C) \to \Cat(\Mod(\mc C),\Mod(\mc D)) \simeq \Cat(\ov{B\mc M},\ov{B\mc N}) \simeq \Mon(\mc M,\mc N), \]
  where the equivalence $\Cat(\Mod(\mc C),\Mod(\mc D)) \simeq \Cat(\ov{B\mc M},\ov{B\mc N})$ depends on the choice of the inverse $I_{\mc D} : \Mod(\mc D) \to \ov{B\mc N}$. 
  In particular, given $\theta : \mc D \to \mc C$, we have a diagram
  \[
  \begin{tikzcd}
    {\mc D} && {\mc C} \\
    & \Fin
    \arrow["\theta", bend left, from=1-1, to=1-3]
    \arrow[""{name=0, anchor=center, inner sep=0}, "{\star_{\mc D}}"', curve={height=16pt}, from=1-1, to=2-2]
    \arrow[""{name=1, anchor=center, inner sep=0}, "\star_{\mc C}\theta", curve={height=-16pt}, from=1-1, to=2-2]
    \arrow["{\star_{\mc C}}", from=1-3, to=2-2]
    \arrow["r_\theta"{description}, curve={height=-6pt}, shorten <=4pt, shorten >=4pt, Rightarrow, from=0, to=1]
    \arrow["s_\theta"{description}, curve={height=-6pt}, shorten <=4pt, shorten >=4pt, Rightarrow, from=1, to=0]
  \end{tikzcd}
  \]
  where as explained the retract pair $(r_\theta,s_\theta)$ is specified by $I_{\mc D}$. The corresponding semigroup morphism $P_\theta : \mc M \to \mc N$ is then given by
  \[ P_\theta(f) \coloneq s_\theta(f \circ \theta)r_\theta. \]
  Since $\ProMon$ is a locally fully faithful sub-2-category of $\Mon$, to show the functor $P_-$ lands in $\ProMon(\mc M,\mc N)$ it suffices to show $P_\theta$ is continuous for any $\theta$. Now consider the clopen $\mc N_g$ of $\mc N$ for some $g : \star_{\mc D}\varphi \to \star_{\mc D}\varphi$ with $\varphi\in\mc D$. By construction,
  \[ (P_\theta)\inv(\mc N_g) = \scomp{f\in\mc M}{s_\varphi \circ f_{\theta\varphi} \circ r_\varphi = g}. \]
  This is a local condition on $\mc M$, in the sense that it is the inverse image $\pi_{\theta\varphi}\inv(U)$ for some subset $U \subseteq \Fin(\star_{\mc C}\theta\varphi,\star_{\mc C}\theta\varphi)$, hence is clopen. Since clopens of the form $\mc N_g$ form a subbasis, it follows that $P_\theta$ is continuous. 

  It remains to establish pseudo-functoriality. The unit can be chosen as identities, since by assumption on the chosen inverse $I_{\mc D}$, $s_\id = r_\id = \id$. Given $\psi : \mc E \to \mc D$ and $\theta : \mc D \to \mc C$, we have a composite retract pair between $\star_{\mc E}$ and $\star_{\mc C}\theta\psi$, given by
  \[ r = (r_\theta \circ \psi) r_\psi, \quad s = s_\psi(s_\theta \circ \psi). \]
  Furthermore, by construction for any $f\in\mc M$, $P_\psi P_\theta(f) = s(f\circ\theta\psi)r$. This way, we may define the compositor $c_{\psi,\theta} : P_\psi P_\theta \nt P_{\theta\psi}$ as 
  \[ c_{\psi,\theta} \coloneq s_{\theta\psi}r. \]
  This 2-cell is invertible with inverse $c_{\psi,\theta}\inv \coloneq sr_{\theta\psi}$: indeed,
  \[ c\inv_{\psi,\theta} c_{\psi,\theta} = sr = (P_\psi P_\theta)(1), \quad c_{\psi,\theta}c_{\psi,\theta}\inv = s_{\theta\psi}r_{\theta\psi} = P_{\theta\psi}(1). \]
  By construction, $c_{\psi,\theta}$ is natural in $\theta$ and $\psi$. The coherence for units is trivial, since by our construction units are strictly identities. The coherence for compositors can be routinely checked.
\end{proof}

When building the pseudo-functor $\PTc\op \to \ProMon$ we have explicitly assumed we have chosen a canonical model for each categorical theory in $\PTc$, and also a retract pair between the canonical model and any other model. The resulting construction of the pseudo-functor is subject to such a choice, and there is no canonical way of building such a pseudo-functor without choice. However, with more care one can show that different choices will induce \emph{pseudo-naturally equivalent} pseudo-functors, but we omit proving this since this is not the primary goal of this paper. On the other hand, we will soon see that there is a 2-functor from profinite monoids to categorical theories which can be canonically constructed, and this will be the focus of the next section.

\section{Classifying Topoi of Profinite Monoids}\label{sec:toposofprofmonoid}

In this section, we will describe a construction from profinite monoids to categorical theories. However, instead of directly constructing a categorical pretopos from a profinite monoid, our approach is to study the \emph{classifying topos} for a profinite monoid $\mc M$. In particular, we will see that the classifying topos construction induces a 2-functor
\[ \Set[-] : \ProMon \to \CTopoi \]
from the 2-category of profinite monoids to the 2-category of coherent topoi and coherent geometric morphisms. Later in~\cref{sec:duality} we will see that this 2-functor is 2-fully faithful, and its image lies in $\PTc\op$ under the equivalence $\CTopoi \simeq \PT\op$.

The classifying topos of a profinite monoid $\mc M$ is the topos $\Set[\mc M]$ of \emph{continuous left $\mc M$-sets}. Explicitly, an object in $\Set[\mc M]$ is a set $X$ with a left $\mc M$-action $\mc M \times X \to X$, which is a continuous map when $X$ is viewed as a discrete space. Morphisms are simply functions that preserve the $\mc M$-action. 

We start by recalling some basic properties of $\Set[\mc M]$; for an extensive investigation of topoi induced by topological monoids, we refer the readers to~\citet{rogers2021toposes}. By forgetting the topology, the monoid $\mc M$ also induces a presheaf topos $[B\mc M,\Set]$ of arbitrary left $\mc M$-sets. We have an evident fully faithful functor $\Set[\mc M] \hook [B\mc M,\Set]$. If $X\in[B\mc M,\Set]$ lies in $\Set[\mc M]$, we will say it is \emph{continuous}. We can characterise continuous left $\mc M$-sets as follows:

\begin{lemma}\label{lem:continuousaction}
  Let $X\in[B\mc M,\Set]$. The following are equivalent:
  \begin{enumerate}
    \item $X$ is continuous.
    \item The subset $I^f_x = \scomp{g}{f\cdot x = g\cdot x}$ for any $x\in X,f\in\mc M$ is (cl)open.
    \item The relation $\operatorname{\sim}_x = \scomp{(f,g)}{f\cdot x = g\cdot x}$ for any $x\in X$ is (cl)open.
  \end{enumerate}
\end{lemma}
\begin{proof}
  This is well-known; for a proof see e.g.~\citet{Rogers2023toposesof}.
\end{proof}

We observe that the above characterisation of when $X\in[B\mc M,\Set]$ is continuous is local on points in $X$, thus it will in particular imply that $\Set[\mc M]$ is closed under subobjects in $[B\mc M,\Set]$. In fact, it is the inverse image part of a hyperconnected geometric morphism $[B\mc M,\Set] \surj \Set[\mc M]$ \citep{Rogers2023toposesof}, thus is also closed under quotients.

Our first goal is to show $\Set[\mc M]$ is a coherent topos, which we do by studying its generating objects. The presheaf topos $[B\mc M,\Set]$ is generated by the unique representable object, which as a left $\mc M$-set is the monoid $\mc M$ itself equipped with the canonical $\mc M$-action. In general, $\mc M$ is not continuous, unless $\mc M$ has the discrete topology. However, it will be true that quotients of $\mc M$ which lie in $\Set[\mc M]$ generate the latter topos. Let us first characterise continuous quotients of $\mc M$:

\begin{lemma}\label{lem:continuousiffcongopen}
  Given any surjection $q : \mc M \surj N$ in $[B\mc M,\Set]$, $N$ is continuous iff the left congruence $\sim_N$ on $\mc M$ given by $f \sim_N g \eff q(f) = q(g)$ is open.
\end{lemma}
\begin{proof}
  Suppose $N$ is continuous. Consider $x = q(1)$. For any $f\in\mc M$, we have
  \[ I^f_x = \scomp{g}{f \cdot x = g \cdot x} = \scomp{g}{q(f) = q(g)} = \scomp{g}{f \sim_N g}. \]
  In particular, we have that
  \[ \operatorname{\sim}_N = \bigcup_{f\in\mc M} I^f_x \times I^f_x, \]
  which implies $\sim_N$ is open. On the other hand, suppose $\sim_N$ is open. For any $x\in N$, let $h$ be such that $q(h) = x$. Then by definition,
  \[ I^f_x = \scomp{g}{f \cdot x = g \cdot x} = \scomp{g}{fh \sim_N gh}. \]
  First observe that the following new left congruence is also open
  \[ f \sim^h_N g \coloneq fh \sim_N gh, \]
  since it is the preimage of $\sim_N$ under the continuous map
  \[ (-) \cdot h \times (-) \cdot h : \mc M \times \mc M \to \mc M \times \mc M. \]
  This way, if $f \sim^h_N g$, then we can find some open neighbourhoods $U_f,U_g$ of $f,g$ that $U_f \times U_g \subseteq \operatorname{\sim}^h_N$. It then follows that
  \[ I^f_x = \bigcup_{f \sim^h_N g} U_g, \]
  which is thus open.
\end{proof}

The above fact applies to any topological monoid. For a profinite monoid $\mc M$, it can be equivalently described as a cofiltered limit of finite monoids \cite[Ch. VI.2]{johnstone1982stone}, thus we can write it as follows, 
\[ \mc M \cong \lt_{i\in I}\mc M_i, \]
where each $\mc M_i$ is finite. Without loss of generality, we may assume the projection maps $\mc M \surj \mc M_i$ are all surjective. In this case the topology on $\mc M$ is induced by identifying it as a subspace $\mc M \subseteq \prod_{i\in I}\mc M_i$ of the product space, and continuous images of $\mc M$ are governed by its finite continuous quotients $\mc M \surj \mc M_i$:

\begin{lemma}\label{lem:continuousquotientsprofinite}
  Let $\mc M \cong \lt_{i\in I}\mc M_i$ be a profinite monoid. For any quotient $\mc M \surj X$ in $[B\mc M,\Set]$, $X$ is continuous iff the quotient factors through some $\mc M \surj \mc M_i \surj X$.
\end{lemma}
\begin{proof}
  Since the congruence $\sim_i$ induced by the quotient $\mc M \surj \mc M_i$ is clopen, by~\cref{lem:continuousiffcongopen} $\mc M_i$ is continuous, which implies any quotients of $\mc M_i$ are also continuous. On the other hand, suppose we have a continuous quotient $\mc M \surj X$ inducing an open equivalence relation $\sim_X$. By~\citet[Prop. VI.2.9]{johnstone1982stone}, $\sim_X$ contains a clopen congruence $\sim_i$, which means we have a factorisation $\mc M \surj \mc M_i \surj X$.
\end{proof}

Recall that an object $X$ in a topos $\mc E$ is \emph{compact} if $X\in\sub_{\mc E}(X)$ is a compact object. Equivalently, any jointly surjective family $\set{Y_i \to X}_{i\in I}$ contains a finite family which is already jointly surjective. 

\begin{proposition}\label{prop:compactobjectsincontinuousaction}
  Let $\mc M \cong \lt_{i\in I}\mc M_i$ be a profinite monoid. For any $X$ in $\Set[\mc M]$, the following are equivalent:
  \begin{enumerate}
    \item $X$ is a compact object in $\Set[\mc M]$;
    \item $X$ has a finite underlying set;
    \item $X$ can be written as a quotient $\coprod_{i\in n}\mc M_i \surj X$.
  \end{enumerate}
\end{proposition}
\begin{proof}
  (3) $\nt$ (2) and (2) $\nt$ (1) are evident. For (1) $\nt$ (3), suppose $X$ is compact. Since $\mc M$ is the only representable in $[B\mc M,\Set]$, there is a jointly epic family of maps from $\mc M$ to $X$. Given any morphism $\mc M \to X$, we may factor it as follows,
  \[
  \begin{tikzcd}
    {\mc M} & X \\
    {\mc M_i} & {X_i}
    \arrow[from=1-1, to=1-2]
    \arrow[two heads, from=1-1, to=2-1]
    \arrow[two heads, from=1-1, to=2-2]
    \arrow[two heads, from=2-1, to=2-2]
    \arrow[tail, from=2-2, to=1-2]
  \end{tikzcd}
  \]
  The factorisation $\mc M \surj X_i \inj X$ is the image factorisation in $[B\mc M,\Set]$. Since $\Set[\mc M]$ is closed under subobjects, $X_i$ is also continuous, and thus the factorisation $\mc M \surj \mc M_i \surj X_i$ is induced by~\cref{lem:continuousquotientsprofinite}. Since $X$ is compact, it is a finite union $X = \bigcup_{i\in n}X_i$, which implies $\coprod_{i\in n}\mc M_i \surj X$.
\end{proof}

In particular,~\cref{prop:compactobjectsincontinuousaction} implies that compact objects in $\Set[\mc M]$ coincide with $\Fin[\mc M]$, the full subcategory of $\Set[\mc M]$ whose objects have finite underlying sets. Recall that an object $X$ is \emph{coherent} in a topos if it is compact, and if any map $Y \to X$ from a compact domain has a compact kernel $Y \times_XY$. A topos is coherent if its full subcategory of coherent objects is left exact and forms a separating family \cite[D3.3]{johnstone2002sketches}.

\begin{proposition}\label{prop:contprofincoherent}
  For any profinite monoid $\mc M$, $\Set[\mc M]$ is a coherent topos, where its coherent objects are exactly $\Fin[\mc M]$.
\end{proposition}
\begin{proof}
  By Yoneda, any object $X\in\Set[\mc M]$ is covered by a family of morphisms from $\mc M$ in $[B\mc M,\Set]$. Then by the same proof as in~\cref{prop:compactobjectsincontinuousaction}, we obtain a cover of $X$ by compact objects of the form $\mc M_i$. This means that $\Fin[\mc M]$ is separating in $\Set[\mc M]$. Hence, it suffices to show $\Fin[\mc M]$ is closed under finite limits, since in this case the kernel of a map $f : Y \to X$ in $\Fin[\mc M]$ lies in $\Fin[\mc M]$, thus will be compact as well. But this holds trivially since finite limits in $\Set[\mc M]$ are computed as in $[B\mc M,\Set]$, and thus they commute with the forgetful functor $\Set[\mc M] \to \Set$. In particular, a finite limit in $\Fin[\mc M]$ again has a finite underlying set, thus is contained in $\Fin[\mc M]$ as well.
\end{proof}

The remaining goal of this section is to describe the 2-functoriality of the construction $\mc M \mapsto \Set[\mc M]$. For any continuous semigroup morphism $\theta : \mc M \to \mc N$ between two profinite monoids, let $e = \theta(1)$. We define the base change functor
\[ \theta^* : \Set[\mc N] \to \Set[\mc M], \]
by sending any continuous $\mc N$-set $X$ to the following $\mc M$-set,
\[ \theta^*X = eX \coloneq \scomp{x\in X}{\exists y\in X.\, e\cdot y = x}. \]
For any $f\in\mc M$ and $x\in\theta^*X$, the action on $\theta^*X$ is simply 
\[ f \cdot x \coloneq \theta(f) \cdot x. \]
Notice that this is a well-defined action: Associativity is evident; for unity,
\[ 1 \cdot x = e \cdot x = e \cdot e \cdot y = e \cdot y = x. \]
Furthermore, $\theta^*X$ is indeed a continuous $\mc M$-set,
\[ I^f_x \coloneq \scomp{g\in\mc M}{f\cdot x = g\cdot x} = \scomp{g\in\mc M}{\theta(f)\cdot x = \theta(g)\cdot x} = \theta\inv(I^{\theta f}_x). \]
Thus, if $\theta$ is continuous, $I^f_x$ must also be open for any $x\in \theta^*X$ and $f\in\mc M$. This shows that $\theta^*$ is well-defined on objects.

For morphisms, suppose we have $u : X \to Y$ in $\Set[\mc N]$. The action of $\theta^*$ on $u$ does not change the underlying function, viz.\ $\theta^*u : \theta^*X \to \theta^*Y$ is simply the restriction of $u$. To show this is well-defined, notice that if $x \in \theta^*X$,
\[ u(x) = u(e\cdot y) = e\cdot u(y) \in \theta^*Y, \]
thus we indeed have a morphism $\theta^*u : \theta^*X \to \theta^*Y$. This gives us a 2-functor from $\ProMon$, the 2-category of profinite monoids, to $\CTopoi$, the 2-category of coherent topoi and coherent geometric morphisms. Recall that a geometric morphism $f : \mc E \to \mc F$ between two coherent topoi is coherent, if the inverse image $f^*$ preserves coherent objects.

\begin{theorem}\label{thm:promonoidtotopos}
  The above construction gives us a 2-functor
  \[ \Set[-] : \ProMon \to \CTopoi. \]
\end{theorem}
\begin{proof}
  The facts that $\theta^*$ is the inverse image part of a geometric morphism, and that any conjugate also induces 2-cells in $\mb{Topoi}$, are established by~\citet[Thm. 5.4.3]{rogers2021toposes}. Since for any $X$, $\theta^*X$ is a subset of $X$, it follows by~\cref{prop:contprofincoherent} that $\theta^*$ restricts to a functor on coherent objects, thus it lands in $\CTopoi$.
\end{proof}

To show the 2-functor in~\cref{thm:promonoidtotopos} describes an equivalence, in particular we need to show that $\Set[-] : \ProMon \hook \CTopoi$ is 2-fully faithful and lands in the full subcategory $\PTc\op \hook \CTopoi$. In fact~\citet[Cor. 5.4.7]{rogers2021toposes} has already shown that this 2-functor is \emph{locally fully faithful}, i.e.\ for any $\mc M,\mc N\in\ProMon$, the induced functor
\[ \ProMon(\mc M,\mc N) \hook \CTopoi(\Set[\mc M],\Set[\mc N]) \]
is fully faithful. In the next section we will prove this is also essentially surjective, thus induces an equivalence.

\section{Profinite Monoids as a Sub-2-Category of Topoi}\label{sec:duality}

In this section we will show that the 2-functor given in~\cref{thm:promonoidtotopos} is in fact 2-fully faithful, and restricts to coherent topoi coming from categorical pretopoi,
\[ \ProMon \hook \CTopoi_c \simeq \PTc\op. \]
It is well-known that the classifying topos construction provides a duality $\PT\op \simeq \CTopoi$, and here we use $\CTopoi_c$ to denote the full subcategory of $\CTopoi$ corresponding to $\PTc$ under this duality. In particular, we may call a coherent topos \emph{categorical}, if its corresponding pretopos is categorical. In fact, we will even show that this 2-functor is still 2-fully faithful when postcomposed with the (non-full) inclusion $\CTopoi \inj \Topoi$, thus we can identify $\ProMon$ as a full sub-2-category of topoi.

To show that a coherent topos $\Set[\mc M]$ for a profinite monoid $\mc M$ is categorical, we start with the case of $\mc M$ being a \emph{finite discrete} monoid:

\begin{lemma}\label{lem:finitemonoid}
    For a finite discrete monoid $\mc M$, $\Fin[\mc M]$ is a categorical pretopos, with
    \[ \Mod(\Fin[\mc M]) \simeq \ov{B\mc M}. \]
    In other words, $\mc M$ can be identified as the endomorphism monoid of the canonical model of $\Fin[\mc M]$. In this case the profinite topology on $\mc M$ given in~\cref{prop:profinitemonoid} is discrete. 
\end{lemma}
\begin{proof}
    By~\cref{prop:contprofincoherent}, the classifying topos of $\Fin[\mc M]$ is exactly $\Set[\mc M]$. Since $\mc M$ is discrete, $\Set[\mc M] \simeq [B\mc M,\Set]$ is a presheaf topos. This way, we have
    \[ \Mod(\Fin[\mc M]) \simeq \Topoi(\Set,\Set[\mc M]) \simeq \mb{Flat}(B\mc M\op,\Set) \simeq \mb{Ind}(B\mc M) \simeq \ov{B\mc M}. \]
    Here the first equivalence holds by definition; the second equivalence is given by Diaconescu's theorem \cite[B3.2]{johnstone2002sketches}; the third equivalence holds by construction of ind-objects; while the last holds since the ind-completion of a finite category is its Cauchy completion \cite[Prop. 2.6]{adamek1994locally}. Hence by~\cref{prop:monoiduptocauchy}, $\Fin[\mc M]$ is categorical, and in fact $\mc M$ is the endomorphism monoid of the canonical model. Any profinite topology on a finite set is discrete, thus the topology on $\mc M$ given in~\cref{prop:profinitemonoid} must be discrete.
\end{proof}

\begin{remark}[Models as idempotents]\label{rem:modelsasidempotents}
    Notice that a functor $B\mc M\op \to \Set$ is equivalently a \emph{right $\mc M$-set}, viz.\ a set equipped with a right $\mc M$-action. Under the equivalence in~\cref{lem:finitemonoid}, an object $\star_e\in\ov{B\mc M}$, equivalently an idempotent $e\in \mc M$, corresponds to the right $\mc M$-set $e\mc M$, equipped with the canonical right $\mc M$-action.
\end{remark}

As a first application, we can characterise the geometric morphisms into $\Set[\mc N]$ for a finite discrete monoid $\mc N$:

\begin{proposition}\label{prop:intodiscrete}
  For any profinite monoid $\mc M$ and finite discrete monoid $\mc N$, we have
  \[ \ProMon(\mc M,\mc N) \simeq \CTopoi(\Set[\mc M],\Set[\mc N]). \]
  In fact this is still an equivalence if we replace $\CTopoi$ with $\mb{Topoi}$.
\end{proposition}
\begin{proof}
  As mentioned at the end of~\cref{sec:toposofprofmonoid}, it suffices to show that any geometric morphism is induced by a continuous semigroup map. By Diaconescu's theorem again,
  \[ \Topoi(\Set[\mc M],\Set[\mc N]) \simeq \mb{Flat}(B\mc N\op,\Set[\mc M]), \]
  where for any geometric morphism $\theta : \Set[\mc M] \to \Set[\mc N]$, the corresponding flat functor is obtained by composing with the Yoneda embedding $B\mc N\op \hook \Set[\mc N] \to \Set[\mc M]$. Concretely, it sends the representable object $\mc N$ to $\theta^*\mc N$.

  Now since $\Set[\mc M]$ has a surjective point, whose inverse image is the forgetful functor
  \[ \geo- : \Set[\mc M] \to \Set, \]
  a functor $F : B\mc N\op \to \Set[\mc M]$ is flat iff $\geo- \circ F$ is flat. But we have classified all flat functors from $B\mc N\op$ to $\Set$ in~\cref{lem:finitemonoid},
  \[ \mb{Flat}(B\mc N\op,\Set) \simeq \ov{B\mc N}. \]
  Thus by~\cref{rem:modelsasidempotents}, we have $\geo{\theta^*\mc N} \cong e\mc N$ for some idempotent element $e\in \mc N$, which is finite. Hence, any such $\theta$ is automatically coherent, and we have
  \[ \Topoi(\Set[\mc M],\Set[\mc N]) \simeq \CTopoi(\Set[\mc M],\Set[\mc N]). \]
  
  Furthermore, by construction $\theta^*\mc N$ has an $\mc M$-action, where now we can define a semigroup homomorphism $\theta : \mc M \to \mc N$ by
  \[ \theta(f) \coloneq f \cdot e. \]
  To show this is a semigroup morphism, consider $f,g \in \mc M$. By definition, $\theta(f) = f \cdot e = ea$ and $\theta(g) = g \cdot e = eb$ for some $a,b\in \mc N$. This way, since acting on the right for any element in $\mc N$ is a morphism of $\mc M$-sets on $\theta^*\mc N$, we have
  \[ \theta(fg) = fg \cdot e = f \cdot (g \cdot e) = f \cdot (eb) = (f \cdot e)b = eab. \]
  By the same reasoning, acting on the right by $e$ gives us
  \[ eae = (f \cdot e)e = f \cdot (ee) = f \cdot e = ea. \]
  It follows that
  \[ \theta(fg) = eab = eaeb = \theta(f)\theta(g), \]
  which implies $\theta$ thus defined is indeed a semigroup morphism. To show $\theta$ is continuous, it suffices to consider inverse images of a point $x \in \theta^*\mc N$. By definition,
  \[ \theta\inv(x) = \scomp{f\in\mc M}{f \cdot e = x}. \]
  If $\theta\inv(x)$ is empty, then evidently it is open. If not, then we can choose $f\in\theta\inv(x)$, and observe that
  \[ I^f_e = \scomp{g\in\mc M}{f \cdot e = g \cdot e = x} = \theta\inv(x), \]
  hence is open as well. 

  It is then straightforward to verify that the geometric morphism induced by the semigroup morphism $\theta : \mc M \to \mc N$ given in~\cref{thm:promonoidtotopos} agrees with $\theta^*$ on the generator $\mc N\in\Set[\mc N]$. This implies they coincide as geometric morphisms since $\Set[\mc N]$ is generated by $\mc N$ under colimits. This implies that 
  \[ \ProMon(\mc M,\mc N) \to \Topoi(\Set[\mc M],\Set[\mc N]) \]
  is essentially surjective, thus gives an equivalence as explained at the end of~\cref{sec:toposofprofmonoid}.
\end{proof}

Using this, we can now show the main result of this section:

\begin{theorem}\label{thm:fullyfaithfulpromonoidtocoh}
  The 2-functor $\Set[-]$ constructed in~\cref{thm:promonoidtotopos} is 2-fully faithful,
  \[ \Set[-] : \ProMon \hook \CTopoi. \]
  It is still 2-fully faithful when postcomposed with the inclusion $\CTopoi \inj \Topoi$.
\end{theorem}
\begin{proof}
  Again, as mentioned at the end of~\cref{sec:toposofprofmonoid}, it suffices to show any geometric morphism is induced by a semigroup morphism between profinite monoids. The crucial observation is that, given any profinite monoid $\mc N$, by viewing it as a cofiltered limit of finite monoids $\mc N \cong \lt_{i\in I}\mc N_i$, we have an induced family of inclusions of categories,
  \[ \set{\Fin[\mc N_i] \hook \Fin[\mc N]}_{i\in I}. \]
  Notice that each inclusion $\Fin[\mc N_i] \hook \Fin[\mc N]$ is indeed fully faithful, with their union being $\Fin[\mc N]$. It follows that $\Fin[\mc N] \cong \ct_{i\in I}\Fin[\mc N_i]$ is the filtered colimit of these pretopoi. By the duality $\mb{PT}\op \simeq \CTopoi$, we have
  \[ \Set[\mc N] \simeq \lt_{i\in I}\Set[\mc N_i], \]
  in the 2-category $\CTopoi$. Thus, we must have
  \begin{align*}
    \CTopoi(\Set[\mc M],\Set[\mc N]) 
    &\simeq \lt_{i\in I}\CTopoi(\Set[\mc M],\Set[\mc N_i]) \\
    &\simeq \lt_{i\in I}\ProMon(\mc M,\mc N_i) \\
    &\simeq \ProMon(\mc M,\lt_{i\in I}\mc N_i) \\
    &\simeq \ProMon(\mc M,\mc N)
  \end{align*}
  In fact, it follows from~\citet[Thm 8.3.13]{SGA4} that cofiltered limits of coherent topoi and coherent geometric morphisms are computed in $\Topoi$, thus the limit $\Set[\mc N] \simeq \lt_{i\in I}\Set[\mc N_i]$ is also a cofiltered limit in $\Topoi$. Also by~\cref{prop:intodiscrete}, any geometric morphism from $\Set[\mc M]$ to $\Set[\mc N]$ is automatically coherent, thus the above equivalences also hold when we replace $\CTopoi$ with $\Topoi$.
\end{proof}

As a corollary, we can show that the embedding $\Set[-] : \ProMon \hook \CTopoi$ indeed factors through $\CTopoi_c$:

\begin{corollary}\label{cor:promoncategorical}
  Any coherent topos of the form $\Set[\mc M]$ is categorical. Concretely, we have
  \[ \Topoi(\Set,\Set[\mc M]) \simeq \ov{B\mc M}. \]
\end{corollary}
\begin{proof}
  Notice that $\Set \simeq \Set[1]$ for the terminal monoid $1$. Hence, by~\cref{thm:fullyfaithfulpromonoidtocoh} we have the equivalence
  \[ \mb{Topoi}(\Set,\Set[\mc M]) \simeq \ProMon(1,\mc M) \simeq \ov{B\mc M}. \]
  The final equivalence is due to the fact that semigroup morphisms from $1$ to $\mc M$ can be identified with idempotents in $\mc M$. Also, conjugates are exactly morphisms in the Cauchy completion.
\end{proof}

\section{Duality for Categorical Theories}\label{sec:duality_for_categorical_theories}

The goal in this section is to complete the proof of the following duality,
\[ \PTc\op \simeq \CTopoi_c \simeq \ProMon. \]
For this, we need to show that the 2-fully faithful embedding $\ProMon \hook \CTopoi_c$ we have obtained in~\cref{sec:duality} is also essentially surjective. We will see in this section that this can be obtained as a corollary of \emph{weak conceptual completeness} of pretopoi.

Recall that~\cref{sec:profmonoidofcattheory} has shown any categorical pretopos $\mc C$ has a canonical profinite monoid $\mc M$ associated to it, where $\mc M$ is the endomorphism monoid for the canonical model $\star$ for $\mc C$. Notice that for any $\varphi\in\mc C$ the canonical model evaluated at $\varphi$, viz.\ the set $\star\varphi$, has a natural left $\mc M$-action, since $\mc M$ by definition is the endomorphism monoid of $\star$. In fact, this lifts to a functor as follows:

\begin{lemma}\label{lem:comparisonfunctor}
  For a categorical pretopos $\mc C$ with canonical model $\star$ and $\mc M \coloneq \Mod(\mc C)(\star,\star)$, there is a coherent functor $U : \mc C \to \Fin[\mc M]$ making the following diagram commute,
  \[
  \begin{tikzcd}
    \mc C \ar[dr, tail, "\star"'] \ar[rr, "U"] & & \Fin[\mc M] \ar[dl, tail, "\geo-"] \\
    & \Fin &
  \end{tikzcd}
  \]
\end{lemma}
\begin{proof}
  For any $x\in\star\varphi$ and any $f\in\mc M$, by definition
  \[ I^f_x \coloneq \scomp{g\in\mc M}{f\cdot x = g\cdot x} = \scomp{g\in\mc M}{f_\varphi x = g_\varphi x}, \]
  which is clopen in $\mc M$. Hence, $\star\varphi$ with the $\mc M$-action induced by endomorphisms of $\star$ belongs to $\Fin[\mc M]$, which we will denote as $U\varphi$. The functor $U$ is coherent because $\star$ preserves the coherent structures and $\geo-$ creates them.
\end{proof}

To show that $U$ is an equivalence of categories, we recall the weak conceptual completeness established in~\citet{alma991006251722103606}: a coherent functor $\theta : \mc C \to \mc D$ between two pretopoi is an equivalence iff $\theta^* : \Mod(\mc D) \to \Mod(\mc C)$ is an equivalence of categories. We use this to show that $U$ in~\cref{lem:comparisonfunctor} is an equivalence:

\begin{proposition}\label{prop:compequivalence}
  For any categorical pretopos $\mc C$ with $\mc M$ the associated profinite monoid, the functor constructed in~\cref{lem:comparisonfunctor} is an equivalence,
  \[ U : \mc C \simeq \Fin[\mc M]. \]
\end{proposition}
\begin{proof}
  Consider the functor on models induced by $U$,
  \[ U^* : \Mod(\Fin[\mc M]) \simeq \ov{B\mc M} \to \Mod(\mc C) \simeq \ov{B\mc M}. \]
  By the commutativity of $U$ given in~\cref{lem:comparisonfunctor}, $U^*$ sends the canonical point in $\ov{B\mc M}$ to the canonical point in $\ov{B\mc M}$, thus $U^*$ is essentially surjective. Fully faithfulness is evident, since by construction endomorphisms on $\star$ and on $\geo-$ coincide when precomposed with $U$. By the weak conceptual completeness, $U$ itself is an equivalence.
\end{proof}

Finally, we can state the main theorem of this paper:

\begin{theorem}\label{thm:main}
  There is a duality between the 2-category of categorical pretopoi and the 2-category of profinite monoids. Furthermore, it restricts to a duality between Boolean categorical pretopoi and profinite groups,
  \[
  \begin{tikzcd}
    \mb{BPT}_c\op \ar[r, "\simeq"] \ar[d, hook] & \mb{ProGrp} \ar[d, hook] \\
    \PTc\op \ar[r, "\simeq"] & \ProMon.
  \end{tikzcd}
  \]
\end{theorem}
\begin{proof}
  By~\cref{thm:fullyfaithfulpromonoidtocoh} and~\cref{cor:promoncategorical}, there is a 2-fully faithful embedding
  \[ \ProMon \hook \CTopoi_c \simeq \PTc\op. \]
  \cref{prop:compequivalence} then implies that this is also essentially surjective, thus an equivalence. Now by~\cref{cor:booleancategorical}, if the categorical theory is Boolean then its associated profinite monoid is in fact a profinite group. On the other hand, for a profinite group $\mc G$, $\Fin[\mc G]$ is indeed a Boolean pretopos. Hence, this induces a duality $\mb{BPT}_c\op \simeq \mb{ProGrp}$.
\end{proof}

\section*{Acknowledgement}

We thank \emph{Morgan Rogers} for carefully reading an earlier draft of this paper and for suggesting many improvements. We also thank an anonymous referee, whose comments improve the presentation and technical precision of this paper.

\section*{Funding}

The author has no relevant funding information to disclose.

\bibliographystyle{apalike}
\bibliography{mybib}

\end{document}